\documentclass[11pt]{amsart}
\usepackage{amsmath}
\usepackage{amssymb}
\usepackage{tabularx}
\usepackage{enumerate}
\usepackage{texdraw}
\usepackage{color}

\usepackage{mathrsfs}
\usepackage{amsfonts,amssymb,amsmath}
\usepackage{epsfig}

\makeatletter
\@addtoreset{equation}{section}

\makeatother
\newtheorem{thm}{Theorem}[section]
\newtheorem{lem}[thm]{Lemma}

\newtheorem{prop}[thm]{Proposition}
\newtheorem{remark}[thm]{Remark}

\newcommand{\R}{\mathbb{R}}

\newcommand{\Z}{\mathcal{Z}}

\begin{document}
\title[Zero Number Diminishing Properties]{The Zero Number Diminishing Property under General Boundary Conditions}
\thanks{$^*$ This research was partly supported by NSFC (No. 11671262)}

\author[B. Lou]{Bendong Lou$^{*}$}
\thanks{$*$ Mathematics \& Science College, Shanghai Normal University, Shanghai 200234, China.}
\thanks{{\bf Email:} {\sf lou@shnu.edu.cn}.}
\date{}

\begin{abstract}
The so-called {\it zero number diminishing property} (or {\it zero number argument}) is a powerful tool in qualitative studies of one dimensional
parabolic equations, which says that, under the zero- or non-zero-Dirichlet boundary conditions, the number of zeroes of the solution
$u(x,t)$ of a linear equation is finite, non-increasing and strictly decreasing when there are multiple zeroes (cf. \cite{Ang}).
In this paper we extend the result to the problems with more general boundary conditions: $u= 0$ sometime and $u\not= 0$ at
other times on the domain boundaries. Such results can be applied in particular to parabolic equations
with Robin and free boundary conditions.
\end{abstract}

\maketitle

%%%%%%%%%%%%%%%%%%%%%%%%%%%%%%%%%%%%%%%%%%%%%%%%%%%%%%%%%%%%%%%%%%%%%%%%%%%%%%%%%%%%%%%%%
%%%%%%%%%%%%%%%%%%%%%%%%%%%%%%%%%%%%%%%%%%%%%%%%%%%%%%%%%%%%%%%%%%%%%%%%%%%%%%%%%%%%%%%%%

\section{Introduction}
Consider the following one dimensional linear parabolic equation:
\begin{equation}\label{linear}
u_t=a(x,t) u_{xx}+b(x,t) u_x+c(x,t) u\quad \mbox{ in } \Omega := \{(x,t) \mid \xi_1 (t) <x < \xi_2 (t),\ t\in (0,T) \},
\end{equation}
where $T>0$ is a constant, $\xi_1$ and $\xi_2$ are continuous functions to be specified below.
For each $t\in (0,T)$, denote by
$$
\Z(t) := \#\{x\in I(t)\mid u(x,t)=0\}
$$
the number of zeroes of $u(\cdot,t)$ in the interval $I(t) := [\xi_1(t), \xi_2(t)]$.
Sometimes we also write $\Z(t)$ as $\Z_{I(t)} [u(\cdot, t)]$ to emphasize the interval $I(t)$.
A point $x_0\in I(t)$ is called a {\it multiple zero} (or {\it degenerate zero}) of
$u(\cdot, t)$ if $u(x_0,t) = u_x (x_0,t)=0$.
In 1988, Angenent \cite{Ang} proved the following properties about the zero number of $u$:

\vskip 3mm

\noindent
{\bf Theorem A} (\cite{Ang}). {\it Assume $\xi_i(t)\equiv \xi^0_i\ (i=1,2)$ are constants with $\xi^0_1 <\xi^0_2$, and
\begin{equation}\label{smoothy}
a, a^{-1}, a_t, a_x, a_{xx}, b, b_t, b_x, c \in L^\infty.
\end{equation}
If $u$ is a nontrivial classical solution of \eqref{linear} satisfying the following boundary conditions:
\begin{equation}\label{bdry}
\left\{
\begin{array}{l}
u (\xi^0_1,t )\equiv 0 \mbox{\ or\ }  u(\xi^0_1,t )\not= 0 \ (\forall\, t\in (0,T)),\\
u (\xi^0_2,t )\equiv 0 \mbox{\  or\ }  u(\xi^0_2,t )\not= 0 \ (\forall\, t\in (0,T)),
\end{array}
\right.
\end{equation}
or
\begin{equation}\label{bdry-Neumann}
u_x (\xi^0_1,t )\equiv u_x (\xi^0_2,t )\equiv 0,
%\quad \mbox{while } a\equiv 1,\ b\equiv 0,
\end{equation}
then
\begin{itemize}
\item[(a)] $\Z(t) <\infty$ for each $t \in (0,T)$;
\item[(b)] $\Z (t)$ is non-increasing in $t\in (0, T)$;
\item[(c)] if, $x_0 \in I$ is a multiple zero of $u(\cdot,s)$ for some $s\in (0, T)$, then $\Z (t_1) > \Z (t_2)$ for
all $t_1, t_2$ satisfying $0 < t_1< s< t_2 <T$.
\end{itemize}
}
\vskip 3mm
\noindent
Roughly speaking, this theorem says that $\Z (t)$ is finite and non-increasing in time $t$. Hence the zero number is also
called by some authors as {\it discrete Lyapunov functional}, which indicates that the solution
becomes more and more simple. We refer the references in \cite{Ang},
in particular, Henry \cite{Henry}, Matano \cite{Ma1982}, etc. for more related results. In the original result in \cite{Ang}, the Neumann case \eqref{bdry-Neumann} was actually treated under the additional conditions $a\equiv 1,\ b\equiv 0$, which are mainly required for the regularity and can be omitted as in \cite{ChXY} by considering solutions with lower regularity.

The boundary conditions in \eqref{bdry} are zero- or non-zero-Dirichlet ones, they may not be true
in the problems with Robin or free boundary conditions, since in such cases the following may occur:
\begin{equation}\label{general-bdry}
u(\xi_i(t),t) =0 \mbox{ for some } t,\qquad   u(\xi_i(t),t) \not= 0 \mbox{ for other }t,\quad i=1,2.
\end{equation}
The main purpose of this paper is to extend Theorem A to such cases.  One special example
is the problem with Robin boundary conditions, for which we will
prove the following result.

\begin{thm}\label{thm:Robin}
Under the assumptions of Theorem A, if the boundary conditions \eqref{bdry} and \eqref{bdry-Neumann} are replaced by (linear and/or nonlinear) Robin ones:
\begin{equation}\label{bdry-Robin}
u_x (\xi^0_1 , t) - g_1 (t,u(\xi^0_1 ,t)) =0,\quad
u_x(\xi^0_2, t) + g_2 (t,  u(\xi^0_2, t)) =0,\quad   t\in (0,T),
\end{equation}
for some functions $g_i\ (i=1,2)$ satisfying $g_i (t,0)\equiv 0$, $p g_i(t,p) \geq 0$ for $p\not= 0$, then the conclusions in Theorem A remain hold.
\end{thm}

This theorem will be proved in the next section as a consequence of our main  result (Theorem \ref{thm:main} below). On the other hand, this result itself is also general since it includes some typical boundary conditions as special examples. 
For example, the linear Robin boundary conditions
$$
u_x (\xi^0_1 , t) - \beta_1(t) u(\xi^0_1 ,t) =0,\quad
u_x(\xi^0_2, t) + \beta_2(t)  u(\xi^0_2, t) =0,\quad   t\in (0,T),
$$
for some $\beta_1(t),\ \beta_2(t)\geq 0$, or the homogeneous Neumann boundary conditions. Note that the linear Robin one can be reduced to the Neumann one by a simple transformation, as shown in \cite{BPS}, and so Theorem A is applied directly. However, the transformation from the nonlinear Robin one  \eqref{bdry-Robin} to the Neumann one is not direct.

Besides the above mentioned special boundary conditions, we will actually consider more general cases,
where the distribution of zeroes of $u(\xi_i(t),t)$ can be much more complicated.
Let $u(x,t)$ be a classical solution of \eqref{linear} in $(0,T)$. We consider the following cases:
$$
{\rm (N)} \hskip 40mm u(\xi_i(t),t)\not= 0 \mbox{ for all } t\in (0,T), \hskip 80mm
$$
this is the non-zero-Dirichlet condition on the boundary curve $x=\xi_i(t)$;
$$
{\rm (Z)} \hskip 35mm  \xi_i(t)\in C^1((0,T)),\ \ u(\xi_i(t),t) = 0 \mbox{ for all } t\in (0,T), \hskip 70mm
$$
this is the zero-Dirichlet condition on the boundary curve $x=\xi_i(t)$.
In other cases where \eqref{general-bdry} holds, we  call $s\in (0,T)$  a
{\it Z-moment of $u(\xi_i(t),t)$} when $ u(\xi_i(s), s)=0$, or a {\it N-moment of $u(\xi_i(t),t)$} when $ u(\xi_i(s), s)\not=0$, and
assume
$$
{\rm (H)} \hskip 15mm
\left\{
\begin{array}{l}
\mbox{the Z-moments of } u(\xi_i(t),t) \mbox{ are isolated, and for each Z-moment }s, \\
\xi'_i (s-0) \mbox{ exists},\  (-1)^i  \cdot u(\xi_i (t),t) \cdot  u_x (\xi_i (s),s)\geq 0 \mbox{ for } 0<s-t\ll 1, \\
 u(\xi_i(t),t)\cdot u(x,s) \geq 0 \mbox{\ \ for\ \ } 0<t-s \ll 1,\  |x-\xi_i(s)| \ll 1.
 \end{array}
 \right.
 \hskip 35mm
$$
We will see below that any solution of \eqref{linear} with zero-Neumann, Robin
or Stefan free boundary conditions satisfies the assumption (H).  Our main result is  the following.

%%%%%%%%%%%%%%%%

\begin{thm}\label{thm:main}
Let $T>0$, $\xi_1(t) < \xi_2(t)$ be two continuous functions in $(0,T)$ and $u(\cdot ,t)\not\equiv 0$
for any $t\in (0,T)$ be a classical solution of \eqref{linear}. Assume \eqref{smoothy} holds.
If one of  {\rm (N)}, {\rm (Z)} and {\rm (H)} holds for $u(\xi_i(t),t)\ (i=1,2)$, then all the conclusions in Theorem A remain hold.
\end{thm}

\begin{remark}\rm
In 1998, the conditions \eqref{smoothy} in Theorem A was weakened by Chen in \cite{ChXY} as
\begin{equation}\label{smoothy1}
a, a^{-1}, a_t, a_x, b, c\in L^\infty.
\end{equation}
In fact, by using the new variables
$$
\alpha(t):= \int_{\xi^0_1}^{\xi^0_2} a(z,t)^{-\frac12} dz,\quad
y(x,t):= \alpha(t)^{-1} \int_{\xi^0_1}^x a(z,t)^{-\frac12} dz,\quad
s(t) := \int_0^t \alpha(t)^{-2}  dt
$$
as in \cite{Ang}, the new unknown $w(y,s) = w(y(x,t),s(t)) :=u(x,t)$ solves
$$
w_s = w_{yy} +\tilde{b}(y,s) w_y + \tilde{c}(y,s) w,\quad y\in (0,1), s \in (0,S),
$$
for some $S>0$, and $\tilde{b},\ \tilde{c}$ depending on $a, a^{-1}, a_t, a_x, b, c$.
Hence, under the assumption \eqref{smoothy1}, $w$ satisfies the inequality $|w_s - w_{yy}|\leq M_1 |w_y| +M_0 |w|$
for some $M_0,\ M_1 >0$.
By the strong unique continuation property for $W^{2,1}_{p, loc}$ solution $w$,
Chen showed in \cite{ChXY} that (a)-(c) of Theorem A hold for $w$, so does for $u$.
Using the results in \cite{ChXY}, we see that the assumption \eqref{smoothy} in both Theorems
\ref{thm:Robin} and \ref{thm:main} can be weakened as \eqref{smoothy1}.
\end{remark}

In 1996, Chen and Pol\'{a}\v{c}ik \cite{ChPo} proved the analogue of Theorem A for radially symmetric
solution in a ball in $\R^N$: $u(x,t)= u(|x|,t) =u(r,t)$ of the problem
\begin{equation}\label{RN-P}
 u_t = \Delta u + c(|x|,t) u =  u_{rr} + \frac{N-1}{r} u_r + c(r,t)u, \quad 0<r< \xi^0, 0<t<T,
\end{equation}
with Dirichlet boundary condition $u(\xi^0, t)\equiv 0$, where $c\in L^\infty$. Using the results in
\cite{ChPo} and using a similar approach as in the next section one can obtain the following result.

\begin{thm}\label{thm:RN}
Let $T>0$, $\xi(t) >0$ be a continuous function in $(0,T)$ and $u(\cdot ,t)\not\equiv 0$
for any $t\in (0,T)$ is a classical solution of \eqref{RN-P}, with $\xi^0$ being replaced by $\xi(t)$. Assume $c\in L^\infty$,  one of  {\rm (N)}, {\rm (Z)} and {\rm (H)} holds for $u(\xi (t),t)$.
Then all the conclusions in Theorem A remain hold.
\end{thm}

\noindent
In particular, when $u$ satisfies \eqref{RN-P} and a Robin boundary condition
$u_r (\xi(t),t) + g( t, u(\xi(t),t))=0$ for $t\in (0,T)$ and some function $g$ satisfying the conditions for $g_i$ in Theorem \ref{thm:Robin},
then (H) holds and the conclusions in Theorem A  hold (see the proof of Theorem \ref{thm:Robin}).

\section{Proof of the Main Theorems}
We first give some preliminary lemmas, and then prove Theorems \ref{thm:main} and \ref{thm:Robin}. The latter will be proved as a consequence of the former. For simplicity we write
$w_i(t):= u(\xi_i(t),t), \ i=1,2$.

\subsection{Some preliminary lemmas}
In this subsection we always consider the special case: $\xi_2(t)\equiv X$ and
$u(X,t)\not= 0$ for all $t\in (0,T)$.  Denote $I(t):= [\xi_1 (t), X]$ as before.

\begin{lem}\label{lem:NNN and ZZZ}
Assume {\rm (Z)} holds for $w_1(t) := u(\xi_1(t),t)$, or $s\in (0,T)$ is a N-moment of $w_1(t)$. Then the conclusions {\rm (a)-(c)} in Theorem A hold in a neighborhood of $s$.
\end{lem}

\begin{proof}
When $s$ is a N-moment of $w_1$, the proof is included in the argument behind  Lemma 2.2 of \cite{BPQ}.
More precisely, we have $w_1(s) = u(\xi_1 (s),s)\not= 0$. Then by continuity we can find $\epsilon>0$ small and $\tilde{x}$ with $ 0< \tilde{x} -
\xi_1 (s) \ll 1$ such that $\xi_1 (t) < \tilde{x} < X$ for $t\in J := [s-\epsilon , s + \epsilon]$, and $u (x,t)\not= 0$ for all $x\in [\xi_1 (t), \tilde{x}]$
and $t\in J$. Therefore, $\Z_{I(t)} [u(\cdot,t)] = \Z_{[\tilde{x}, X]} [u(\cdot,t)] $
for all $t\in J$. The conclusions then follow from Theorem A.

Next we consider the case where (Z) holds for $w_1(t)$. Using the new variable
\begin{equation}\label{variable change}
y= \frac{x-\xi_1 (t)}{X - \xi_1 (t)} \ \ \Leftrightarrow \ \ x= \varphi(y,t) :=  \xi_1 (t) + (X -\xi_1 (t)) y
\end{equation}
we see that the new unknown $v(y,t):= u \big( \varphi(y,t), t\big)$ solves
\begin{equation}\label{eq-w}
 v_t = \tilde{a}(y,t) v_{yy} + \tilde{b}(y,t) v_y + \tilde{c}(y,t) v, \quad  0<y<1,\ t\in (0,T),
\end{equation}
with
\begin{equation}\label{tilde a cond}
\left\{
\begin{array}{l}
\displaystyle \tilde{a}(y,t) := \frac{a(\varphi(y,t), t)}{( X - \xi_1 (t))^2 }
\mbox{ satisfying } \tilde{a},\ \tilde{a}^{-1},\ \tilde{a}_t,\ \tilde{a}_y\in L^\infty, \\
\displaystyle  \tilde{b}(y,t) := \frac{\xi'_1 (t) - \xi'_1(t) y }{ X - \xi_1 (t) }
+\frac{b(\varphi (y,t),t)} {X -\xi_1 (t)} \in L^\infty_{loc} ([0,1]\times (0,T)) ,\\%, \ \ \mbox{locally in } t\in (0,T),\\
\tilde{c}(y,t) := c(\varphi (y,t),t) \in L^\infty.
\end{array}
\right.
\end{equation}
Here (and only here) we use  the assumption $\xi_1\in C^1$ in (Z).  Since
$$
v(0,t) = u(\xi_1 (t),t)=0  \mbox{ and } w(1,t)\not= 0 \mbox{ for }t\in (0,T),
$$
we see that the conclusions in Theorem A hold for $v$. So they also hold for $u$ since $\Z_{I(t) } [u(\cdot,t)] = \Z_{[0,1]}[v(\cdot,t)]$.
\end{proof}

From this lemma we see that on each side of $\Omega$, $u$ can not identically
to be {\it multiple zeroes} in a time interval. Otherwise, the number of its zeroes
decreases strictly infinitely many times.  This is impossible.
Now we show that the number of zeroes
of $u$ in $I(t) $ is finite from the {\it very beginning}.

\begin{lem}\label{lem:Z<infty}
There exists a sequence $\{s_k\}\subset (0,T)$ decreasing to $0$ such that
$\Z_{I(s_k)} [u(\cdot, s_k)]$ $ <\infty$.
\end{lem}

\begin{proof}
In case $w_1$ satisfies (Z), the conclusion is proved in the previous lemma.
In case $w_1$ has a sequence of N-moments $\{s_k\}$ decreasing to $0$,
Lemma \ref{lem:NNN and ZZZ} also implies that  $\Z_{I(s_k)}[u(\cdot, s_k)] <\infty$.
\end{proof}

Now we consider a Z-moment $s$ of $w_1 (t)= u(\xi_1 (t),t)$.  When it is isolated, we have
$w_1(s)=0$ but $w_1(t)\not= 0$ in short periods before and after $s$. We will show that, under
the condition (H),  the zero $\xi_1(s)$ of $u(\cdot, s)$ is not a new one, it actually comes from an interior zero of
$u(\cdot, t)$ for $t<s$.  Again we consider the case for $u(\xi_2(t),t)= u(X,t)\not= 0$.

\begin{lem}\label{lem:NZ*}
Assume that $w_1$ satisfies {\rm (H)} and $s$ is one of its  isolated zeroes. Then
\begin{equation}\label{NZ*-non-increasing}
\Z(s) \leq \Z(t) < \infty \mbox{\ \ for\ }t \mbox{ satisfying } 0<s-t\ll 1.
\end{equation}
\end{lem}

\begin{proof}
(1)  Denote by $J_s = (t_0,s)$ for some $t_0\in [0,s)$ an interval containing only N-moments of $w_1$.
By the previous lemmas,
$\Z (t)<\infty$ for all $t\in J_s$,  and there exist at most finitely many values of $t\in J_s$ such that $u(\cdot,t)$
has multiple zeroes in the interior of $I(t)$. Thus we can find a small $\epsilon>0$ such that for any $t\in J_1:=
(s-\epsilon , s) \subset J_s$ , $u(\cdot,t)$ has only simple zeroes  in the closed interval $I(t)$.
Due to the simplicity, the zeroes  of $u(\cdot,t)$ for  $t\in J_1$ can be expressed as
smooth curves:
 \[
 \mbox{$x=\gamma_1(t),..., x=\gamma_m(t)$, with $ \xi_1(t) <\gamma_k (t)<\gamma_{k+1}(t)< X$ for $k=1,..., m-1$.}
 \]

\medskip
(2) For each $k\in\{1,..., m\}$, we show the existence of the limit of $\gamma_k(t)$ as $t\nearrow s$. Clearly
 \[
 \mbox{ $x_k :=\liminf_{t\nearrow s}\gamma_k (t)\geq \xi_1 (s)$ \ \ and \ \ $x_k^*:=\limsup_{t\nearrow s}\gamma_k (t)< X$.}
  \]
  If $x_k< x_k^*$, then $u(x,s)\equiv 0$ for $x\in [x_k, x_k^*]$.
  We may now apply \cite[Theorem 2]{F} or the results in \cite[section 3]{ChXY} to $u$
  over the region $\{(x,t) \mid \xi_1(t)\leq x \leq X, s-\epsilon_1 \leq t \leq s\}$,
  with $\epsilon_1 >0$ sufficiently small, to conclude that $u(x,s)\equiv 0$ for $x\in  I(s)$,
  contradicting our assumption $u(\cdot,t)\not\equiv 0$ for any $t\in (0,T)$. Therefore $x_k:=\lim_{t\nearrow s}\gamma_k (t)$ exists for every $k\in \{1,..., m\}$.

\medskip
(3) To show $x_1= \xi_1(s)$. Assume without loss of generality that $u(\xi_1(t),t)>0$ for
$t\in (s-\epsilon, s)$. Arguing indirectly we assume that $\xi_1 (s) < x_1$. Then in the region $D_1
:=\{(x,t) \mid \xi_1(t) < x <\gamma_1 (t), s-\epsilon  <t \leq s\}$, we have
$u(x,t)>0$ by the maximum principle. Since $u(\xi_1(s), s)=0$ and since $\xi'_1(s-0) $ exists by (H), we can apply
the Hopf boundary lemma to deduce that $u_x (\xi_1(s),s) >0$, contradicting our assumption (H).
This proves $\xi_1(s)= x_1$.

\medskip
(4) Using the strong maximum principle in the region $D_k :=\{(x,t): \gamma_k(t)<x<
\gamma_{k+1}(t),\;  t\in J_1\}$, we conclude that $u(x,s)\not=0$ for $x\in (x_k, x_{k+1})$ when $x_k < x_{k+1}$.

\medskip
(5) From (3) and (4) we conclude that $\Z(s) \leq \Z (t) = m$ for $t$ with $0< s-t \ll 1 $.
This means that \eqref{NZ*-non-increasing} holds for such $t$.
\end{proof}

Next we show that, under the assumption (H),   the number of zeroes  of $u$ decreases strictly after each Z-moment.

\begin{lem}\label{lem:*ZN}
Under the same assumptions as in  the previous lemma. Then
\begin{equation}\label{*ZN-strict-decrease}
\Z (s)  > \Z (t) \mbox{ for }t \mbox{ with } 0<t -s \ll 1.
\end{equation}
\end{lem}

\begin{proof}
Denote by $J'_s = (s, t^0)$ for some $t^0>s$ an interval containing only N-moments of $w_1$.
Since $\Z (s)  <\infty$ by \eqref{NZ*-non-increasing}, we write $\xi_1 (s) = z_1 < z_2 <...< z_n < X$
as the different zeroes of $u(\cdot, s)$ in $I(s):= [\xi_1(s), X]$.
Denote $z^*=(z_1 +z_2)/2$. We claim that there exists a small $\epsilon >0$ such that
$u(x,t)\not=0$ in the domain $D := \{(x,t) \mid \xi_1 (t)\leq x \leq z^*, s<t< s +\epsilon \}$.
In fact, $u(z^*, s)\not=0$. For definiteness, we assume that $u(z^*,s)>0$. Hence
$u(x,s)>0$ for $x\in (\xi_1(s), z^*]$, and, by continuity, there exists $\epsilon \in (0, t^0 -s)$
such that $u(z^*,t)>0$ for $t\in [s, s+\epsilon]$. Using the assumption (H) we have
$w_1 (t) = u(\xi_1(t), t) > 0$ in $(s, s+\epsilon]$, provided $\epsilon $ is sufficiently small.
Then our claim follows by using the maximum principle in the region $D$.
This claim means that the zero $\xi_1(s)$ of $u(\cdot, s)$ will disappear in $D$.
On the other hand, in the region $[z^*,X]\times [s,s+\epsilon]$, the zero number diminishing property
follows from Theorem A. This proves \eqref{*ZN-strict-decrease}.
\end{proof}

\subsection{Proof of the main results}
{\it Proof of Theorem \ref{thm:main}}.
For any $t_0\in (0,T)$, by our assumption $u(\cdot, t_0)\not\equiv 0$ we can find
$X= X(t_0)\in (\xi_1(t_0), \xi_2(t_0))$ with $u(X, t_0)\not=0$. By continuity, there exists
$\epsilon = \epsilon(X,t_0)$ such that
$$
\xi_1(t) < X < \xi_2(t) \mbox{ and } u(X,t)\not=0 \mbox{ for all } t\in J_\epsilon := [t_0 -\epsilon, t_0+\epsilon].
$$
Therefore, to prove the conclusions  (a)-(c) of Theorem A  near $t_0$,
we need only to prove them in two domains
$\Omega_1 (t_0,\epsilon) := \{(x,t)\mid \xi_1 (t)\leq x\leq X, t\in J_\epsilon \}$
and
$\Omega_2 (t_0,\epsilon) := \{(x,t)\mid X\leq x \leq \xi_2 (t), t\in J_\epsilon \}$,
respectively.
In $\Omega_1 (t_0,\epsilon)$, the conclusions have been proved by the lemmas in the previous subsection.
In $\Omega_2 (t_0, \epsilon)$, the proof is similar.
Finally, the conclusions hold in the whole time span $(0,T)$ since $t_0\in (0,T)$ is arbitrarily chosen.  This proves Theorem \ref{thm:main}.
\hfill $\Box$
\bigskip

\noindent
{\it Proof of Theorem \ref{thm:Robin}}.  As above we need only consider the case where $u(\xi^0_2, t)\not=0$ for
$t\in (0,T)$. If $u(\xi^0_1, t)\equiv 0$ or $u(\xi^0_1, t)\not= 0$ for all $t\in (0,T)$, then the conclusions follow from Theorem A. We now consider the other cases,  and show that the boundary conditions \eqref{bdry-Robin} actually imply (H).

Assume $t_1\in (0,T)$ is a N-moment of $u(\xi^0_1, t)$ and $s\in (t_1,T)$ is the smallest Z-moment bigger than $t_1$.
Then $u_x(\xi^0_1, s) = g_1(s, u(\xi^0_1,s)) = g_1(s,0) = 0$ and $J_1 := (t_1, s)$ is an interval consisting of N-moments of $u(\xi^0_1,t)$. Using Lemma A  we have $\Z(t)<\infty$ for $t\in J_1$. Noting $u_x(\xi^0_1, s) =0$  and using Lemma \ref{lem:NZ*} (as well as its proof) we have
$\Z(s)\leq \Z(t)<\infty$ for $t\in J_1$. Thus, the zeroes of $u(\cdot,s)$ in $[\xi^0_1, \xi^0_2]$ are isolated. Then applying the maximum principle in the region $D:= [\xi^0_1, \xi^0_1 +\epsilon]\times [s,s+\epsilon]$ (for some small $\epsilon$) as in the proof of Lemma \ref{lem:*ZN}, and combining with the first  boundary condition in \eqref{bdry-Robin} we conclude that $u$ remains positive or negative in $D\backslash\{(\xi^0_1, s)\}$. Therefore,  the Z-moment $s$ of $w_1(t):= u(\xi^0_1,t)$ is isolated,
and so the last condition in (H) holds: $u(\xi^0_i,t)\cdot u(x,s) > 0$ for  $ 0<t-s \ll 1$ and $ |x-\xi^0_1 | \ll 1$. Thus, the condition (H) is derived from \eqref{bdry-Robin}, and so Theorem \ref{thm:Robin} follows from Theorem \ref{thm:main}. \hfill $\Box$

%%%%%%%%%%%%%%
%%%%%%%%%%%%%%

\section{An Example}
We now use our results to consider the following problem with nonlinear Robin and free boundary conditions:
\begin{equation}\label{ex-p}
\left\{
\begin{array}{ll}
u_t = a(x,t) u_{xx} + b(x,t)u_x + f(x,t,u), & 0<x<h(t), \ t>0,\\
u_x(0,t) = g(t, u(0,t)) , & t>0,\\
u(h(t),t)=0,\ h'(t)=- u_x(h(t),t), & t>0,\\
h(0)=h_0, \ u(x,0)=u_0(x)\geq 0 , & 0\leq x\leq h_0 ,
\end{array}
\right.
\end{equation}
Assume, with $\Omega := [0,\infty)\times (0,\infty)$, $\alpha \in (0,1)$,
$$
{\rm (H1)}\hskip 3mm
\left\{
 \begin{array}{l}
 a, a^{-1}, a_t, a_x, b \in L^\infty(\Omega),\ a, b\in C^\alpha_{loc} (\Omega),\ f\in C^1 (\R^3),\ g\in C^1(\R^2),\\
 f(x,t,0)\equiv 0 \mbox{ and } f(x,t,u)<0 \mbox{ for } t, x\in \R \mbox{ and } u > 1,\ g_u(t,u)\geq 0 \mbox{ for }t\in \R
 \mbox{ and } u\geq 0.
 \end{array}
\right. \hskip 30mm
$$
Reaction diffusion equations with such Stefan free boundary conditions are used to model the spreading of a new species and have been studied in \cite{DGP, DuLin} etc.

\begin{prop}\label{prop:ex1}
Assume {\rm (H1)} and that $a,b,f,g$ are all $T$-periodic in $t$. Then when $u_0\in C([0,h_0])$, the solution $u(x,t)$ of \eqref{ex-p} exists globally. As $t\to \infty$, $h(t)$ increases and converges to $h_\infty\in (h_0, \infty]$ and
$u$ converges to a time periodic solution in the topology $C^{2,1}_{loc}
([0,h_\infty)\times [0,T])$.
\end{prop}

\begin{proof}
By comparison it is easily seen that $0< u(x,t)\leq 1+\|u_0\|_{L^\infty}$ for $t>0,\ x\in [0,h(t)]$. This implies the time global existence of the solution of \eqref{ex-p} (cf. \cite{DuLin}).
For any $s\in (-T,T)$ we compare $u_1(x,t;s):= u(x,t-s)$ (which is defined for $x\in [0,h(t-s)]$) with
$u_2(x,t;s) := u(x,t-s+T)$ (which is defined for $x\in [0, h(t-s+T) ]$, with $h(t-s+T) >h(t-s)$ by the
free boundary condition and the Hopf lemma). Define $\eta^s(x,t) := u_2 (x, t;s) -u_1(x,t;s)$, then
\begin{equation}\label{eq-eta}
\left\{
 \begin{array}{ll}
 \eta^s_t = a(x,t-s)\eta^s_{xx} +b(x,t-s)\eta^s_x +c(x,t,s)\eta^s, & 0<x<h(t-s),\ t>s,\\
 \eta^s_x  (0,t)= g_u (t-s, \theta)  \eta^s (0,t),\quad
 \eta^s  (h(t-s),t) >0 ,\quad   & t>s,
\end{array}
\right.
\end{equation}
where $\theta = \theta_1 u_1 (0,t;s) + (1-\theta_1) u_2 (0,t;s)$ for some $\theta_1\in (0,1)$,
and
$$
c (x,t,s)  =  [f (x,t-s, u_2 ) - f (x,t-s, u_1) ]/(u_2 -u_1)\quad \mbox{for } u_1 \not= u_2,
$$
and extended continuously to the case where $u_1 = u_2$. Note that $a, b, c$ satisfies the
conditions in \eqref{smoothy1} by the hypothesis (H1) and the boundedness of $u$.
Since $\eta^s$ satisfies a Robin condition at $x=0$ and satisfies the non-zero-Dirichlet boundary condition (N) on the right boundary
$x=h(t-s)$, using the zero number diminishing property in our main results, we conclude that $\Z_{[0,h(t-s)] }[\eta^s (\cdot,t)]$ is finite, non-increasing and
decreasing strictly at most finitely many times. Hence, for large $t$,
$\eta^s (\cdot, t)$ has fixed number of zeroes and all of them are simple ones.
In the rest, one can use a similar argument as in \cite{BPS, Lou} to prove the conclusion.
\end{proof}

\begin{remark}\rm
In this proof,  the right end point of $u_1$ is always smaller than that of $u_2$. Actually, for any two solutions
$u_3$ (over $x\in [0,h_3(t)]$) and $u_4$ (over $x\in [0,h_4(t)]$) of the problem \eqref{ex-p} with different initial data, $h_3(t)-h_4(t)$ may change sign many times. Using the Stefan free boundary condition, one can show by a careful analysis that $u_3 -u_4$ satisfies
(H) on the domain boundary $x= \min\{h_3(t),h_4(t)\}$. In this sense we say that the zero number diminishing property in Theorem \ref{thm:main} is applied to one dimensional parabolic equations with Stefan free boundary conditions.
\end{remark}
%%%%%%%%%%%%%%%

{\bf Acknowledgement.} The author would like to thank the anonymous referees for valuable comments and the suggestion for the reference \cite{BPQ}.


\begin{thebibliography}{99}
\bibitem{Ang}
S. B.~Angenent,
{\em The zero set of a solution of a parabolic equation}, J. reine angew. Math., 390 (1988), 79-96.

\bibitem{BPQ}
T. Bartsch, P. Pol\'{a}\v{c}ik and P. Quittner,
{\em Liouville-type theorems and asymptotic behavior of nodal radial solutions of semilinear heat equations}, J. Eur. Math. Soc., 13 (2011), 219-247.


\bibitem{BPS}P. Brunovsk\'{y}, P. Pol\'{a}\v{c}ik and B. Sandstede,  {\em Convergence in general periodic parabolic equations in one space dimension},  Nonl. Anal.,  18 (1992), 209-215.

\bibitem{ChXY} X. Y. Chen,
{\em A strong unique continuation theorem for parabolic equations}, Math. Ann.,  311 (1998), 603-630.

\bibitem{ChPo}X. Y. Chen and P. Pol\'{a}\v{c}ik,
{\em Asymptotic periodicity of positive solutions of reaction diffusion equations on a ball},
J. reine angew. Math., 472 (1996), 17-51.

\bibitem{DGP} Y. Du, Z.M. Guo and R. Peng,
{\em A diffusion logistic model with a free boundary in time-periodic environment}, J. Funct. Anal., 265 (2013), 2089--2142.

\bibitem{DuLin} Y.~Du and Z.~Lin, {\em Spreading-vanishing dichtomy in the diffusive logistic model with a free boundary}, SIAM J. Math. Anal., 42 (2010), 377--405.


\bibitem{F}F. J. Fernandez,
{\em Unique continuation for parabolic operators II},
Comm. Partial Differential Equations, 28 (2003), 1597-1604.

\bibitem{Henry} D. Henry,
{\em Some infinite dimensional Morse Smale systems defined by parabolic differential
equations}, J. Differential Equations, 59 (1985), 165-205.

\bibitem{Lou}
B. Lou, {\em Convergence in time-periodic quasilinear parabolic equations in one space dimension}, J. Differential Equations,
265 (2018), 3952-3969.

\bibitem{Ma1982}H. Matano, {\em Nonincrease of the lap-number of a solution for a one-dimensional semilinear parabolic equation}, J. Fac. Sci. Univ. Tokyo Sect. IA Math., 29 (2) (1982), 401-441.
\end{thebibliography}
\end{document}